\documentclass[11pt, reqno, psamsfonts]{amsart}
\pdfoutput=1

\usepackage{amssymb}
\usepackage{amsthm}
\usepackage{amsmath}
\usepackage{latexsym}
\usepackage[T1]{fontenc}
\usepackage[utf8]{inputenc}
\usepackage[russian, french, english]{babel}

\usepackage{graphicx}
\usepackage{wrapfig}
\usepackage{mathtools}
\usepackage{tikz}
\usepackage{amsbsy}
\usepackage[inline]{enumitem}
\usepackage{mathrsfs}
\usetikzlibrary{shapes,snakes}
\usetikzlibrary{arrows.meta}
\usetikzlibrary{decorations.pathmorphing}
\usetikzlibrary{patterns}
\usetikzlibrary{calc}
\usepackage{float}
\usepackage{array}
\usepackage{subcaption}
\usepackage{multicol}
\usepackage{stmaryrd}
\usepackage{cancel} 
\usepackage{algorithm}
\usepackage{algpseudocode}
\makeatletter
\algnewcommand{\LineComment}[1]{\Statex \hskip\ALG@thistlm {\color{gray}\texttt{// #1}}}
\makeatother

\counterwithin{algorithm}{section}
\usepackage[justification=centering, labelfont=bf]{caption}

\usepackage{lmodern}
\usepackage{mathabx}

\usepackage{upgreek}
\usepackage{titlesec}
\usepackage{bbm}
\usepackage[spacing=true,kerning=true,babel=true,tracking=true]{microtype}
\usepackage[shortcuts]{extdash}
\usepackage[foot]{amsaddr}
\usepackage[left=1in,right=1in,top=1in,bottom=1in,bindingoffset=0cm]{geometry}
\usepackage{bm}
\usepackage{centernot}
\usepackage{mdframed}
\usepackage{hyperref}
\usepackage{mathdots}
\usepackage{xspace}
\hypersetup{
    colorlinks=true,
    linkcolor=blue,
    filecolor=magenta,      
    urlcolor=red,
    citecolor=gray,
}
\allowdisplaybreaks

\usepackage[backend=biber, style=alphabetic, sorting=nyt, maxnames=100,backref=true]{biblatex}
\title{Multigraph edge-coloring with local list sizes}
\date{}
\author{\lsstyle Abhishek~Dhawan}
\email{abhishek.dhawan@math.gatech.edu}
\address{\textls{\normalfont{}School of Mathematics, Georgia Institute of Technology, Atlanta, GA, USA}}

\newtheoremstyle{bfnote}%
{}{}%
{\slshape}{}%
{\bfseries}{\bfseries.}%
{ }%
{\thmname{#1}\thmnumber{ #2}\thmnote{ \ep{\normalfont{}#3}}}

\theoremstyle{bfnote}
\newtheorem{theo}[equation]{Theorem}
\newtheorem*{theo*}{Theorem}
\newtheorem{prop}[equation]{Proposition}
\newtheorem{Lemma}[equation]{Lemma}
\newtheorem{claim}{Claim}[equation]
\newtheorem{corl}[equation]{Corollary}

\newtheorem*{corl*}{Corollary}

\newcounter{ForClaims}[section]

\theoremstyle{definition}
\newtheorem{defn}[equation]{Definition}
\newtheorem*{defn*}{Definition}

\newtheorem*{exmp*}{Example}

\theoremstyle{remark}
\newtheorem*{ques*}{Question}
\newtheorem*{remk*}{Remark}

\newcommand*{\myproofname}{Proof}
\newenvironment{claimproof}[1][\myproofname]{\begin{proof}[#1]}{\end{proof}}

\newcommand{\0}{\emptyset}
\newcommand{\set}[1]{\{#1\}}
\newcommand{\N}{{\mathbb{N}}}
\newcommand{\Z}{\mathbb{Z}}

\renewcommand{\epsilon}{\varepsilon}
\newcommand{\eps}{\epsilon}
\renewcommand{\phi}{\varphi}
\renewcommand{\theta}{\vartheta}
\renewcommand{\leq}{\leqslant}
\renewcommand{\geq}{\geqslant}
\newcommand{\defeq}{\coloneqq}

\newcommand{\bemph}[1]{{\normalfont#1}} 
\newcommand{\ep}[1]{\bemph{(}#1\bemph{)}} 

\newcommand{\emphdef}[1]{\textbf{\textit{{#1}}}}

\newcommand{\emphd}[1]{\emphdef{#1}}

\newcommand{\blank}{\mathsf{blank}}

\numberwithin{equation}{section}

\newcommand{\Sha}{\lfloor3\Delta/2\rfloor}
\newcommand{\Shav}[1]{\lfloor3\deg(#1)/2\rfloor}
\newcommand{\Shift}{\mathsf{Shift}}
\newcommand{\vend}{\mathsf{vEnd}}
\newcommand{\Pot}{\Phi}
\newcommand{\poly}{\mathsf{poly}}
\newcommand{\length}{\mathsf{length}}
\newcommand{\End}{\mathsf{End}}
\newcommand{\Start}{\mathsf{Start}}
\newcommand{\Pivot}{\mathsf{Pivot}}
\newcommand{\vstart}{\mathsf{vStart}}

\titleformat{\section}[block]{\bfseries\scshape\filcenter}{\thesection.}{1ex}{}
\titleformat{\subsection}[block]{\bfseries\filcenter}{\thesubsection.}{1ex}{}
\titleformat{\subsubsection}[runin]{\itshape}{\bfseries\upshape\thesubsubsection.}{1ex}{}[.---]

\titlespacing*{\section}{0pt}{*3}{*1}
\titlespacing*{\subsection}{0pt}{*3}{*1}
\titlespacing*{\subsubsection}{0pt}{*1.5}{*0}

\renewbibmacro{in:}{}

\renewbibmacro*{volume+number+eid}{%
	\printfield{volume}%
	\setunit*{\addnbspace}
	\printfield{number}%
	\setunit{\addcomma\space}%
	\printfield{eid}}

\DeclareFieldFormat[article]{volume}{\textbf{#1}\space}
\DeclareFieldFormat[article]{number}{\mkbibparens{#1}}

\DeclareFieldFormat{journaltitle}{#1,}
\DeclareFieldFormat[thesis]{title}{\mkbibemph{#1}\addperiod}
\DeclareFieldFormat[article, unpublished, thesis]{title}{\mkbibemph{#1},}
\DeclareFieldFormat[book]{title}{\mkbibemph{#1}\addperiod}
\DeclareFieldFormat[unpublished]{howpublished}{#1, }

\DeclareFieldFormat{pages}{#1}

\DeclareFieldFormat[article]{series}{Ser.~#1\addcomma}

\setlist{topsep=3pt,itemsep=3pt}

\begin{document}

\vspace*{0pt}

\maketitle
\begin{abstract}
    Let $G$ be a multigraph and $L\,:\,E(G) \to 2^\N$ be a list assignment for the edges of $G$. Suppose additionally, for every vertex $x$, the edges incident to $x$ have at least $f(x)$ colors in common. We consider a variant of local edge-colorings wherein the color received by an edge $e$ must be contained in $L(e)$. The locality appears in the function $f$, i.e., $f(x)$ is some function of the local structure of $x$ in $G$. Such a notion is a natural generalization of traditional local edge-coloring. Our main results include sufficient conditions on the function $f$ to construct such colorings. As corollaries, we obtain local analogs of Vizing and Shannon's theorems, recovering a recent result of Conley, Greb\'ik, and Pikhurko.
\end{abstract}

\noindent

\section{Introduction}\label{section:intro}

All multigraphs considered in this paper are finite, undirected, and loopless.
Edge-coloring is an important area of graph theory with significant theoretical results as well as extensive applications in computer science and engineering.
For an introduction to the topic, we refer the reader to the books by Bollob{\'{a}}s \cite{Bollobas} and by Stiebitz, Scheide, Toft, and Favrholdt \cite{stiebitz2012graph}.
Before we state our results, we make a few definitions.
For $r \in \N$, we let $[r] \defeq \set{1, \ldots, r}$.
Let $G$ be a multigraph with $n$ vertices and $m$ edges.
For $x,y \in V(G)$ and $e \in E(G)$, we define $E_G(x)$ to be the set of edges incident to $x$, $E_G(x,y)$ to be the set of edges between $x$ and $y$, $N_G(x)$ to be the set of vertices $z$ such that $E_G(x,z) \neq \0$, and $V(e)$ to be the set of endpoints of $e$.
Furthermore, we define the degree of $x$ to be $\deg(x) \defeq |E_G(x)|$, the multiplicity of $\set{x,y}$ to be $\mu(x,y) \defeq |E_G(x, y)|$, and define the maximum degree and multiplicity (denoted $\Delta(G)$ and $\mu(G)$, respectively) as the maximum value attained by these parameters.
Finally, we let $\delta(G)$ denote the minimum degree of a vertex in $G$.

\begin{defn}\label{defn:edge_coloring}
    A \emphd{proper $r$-edge-coloring} of $G$ is a function $\phi\,:\,E(G) \to [r]$ such that $\phi(e) \neq \phi(f)$ whenever $V(e) \cap V(f) \neq \0$ and $e \neq f$.
    The \emphd{chromatic index} of $G$, denoted by $\chi'(G)$, is the minimum $r$ such that $G$ admits a proper $r$-edge-coloring.
\end{defn}

Vizing famously proved $\chi'(G) \leq \Delta(G) + \mu(G)$ \cite{Vizing} (this was also proved independently by Gupta \cite{gupta1966chromatic}).
Prior to that, Shannon proved the upper bound $\Sha$ \cite{Shannon}, which is independent of $\mu$.
These results are based on global parameters ($\Delta$ and $\mu$).
In this paper, we are interested in \textit{local} colorings, i.e., where the set of available colors for an edge $e$ depends on local parameters (such as the degree of its endpoints).
Such a notion of colorings was first introduced by Erd\H{o}s, Rubin, and Taylor in their work on list colorings of the vertices of a graph \cite{erdos1979choosability}.
We are interested in an edge-coloring variant, specifically list edge-coloring.

\begin{defn}[$L$-edge-coloring]\label{defn:local_sha}
    Given a multigraph $G$ and a function $L \,:\,E(G) \to 2^{\N}$ (referred to as a \emphd{list assignment} for the edges of $G$), we say $\phi\,:\,E(G) \to \N$ is a \emphd{proper $L$-edge-coloring} if it is a proper edge-coloring such that for each edge $e \in E(G)$, we have $\phi(e) \in L(e)$.
\end{defn}

The list edge-coloring conjecture is a famous open problem posed independently by a number of researchers in the 1970s and 1980s.
It states that as long as $|L(e)| \geq \chi'(G)$ for each edge $e$, $G$ admits a proper $L$-edge-coloring.
In \cite{borodin1997list}, Borodin, Kostochka, and Woodall proved a local list version of K\H{o}nig's theorem, providing an alternate proof of the list edge-coloring conjecture for bipartite multigraphs.
Specifically, they showed that every bipartite multigraph admits a proper $L$-edge-coloring provided that for each edge $e \in E_G(x, y)$, $|L(e)| \geq \max\set{\deg(x), \deg(y)}$.
They also proved an analogous result for Shannon's bound.
Specifically, there is a proper $L$-edge-coloring for a multigraph $G$ if for each edge $e \in E_G(x, y)$, we have
\[|L(e)| \geq \max\set{\deg(x), \deg(y)} + \lfloor\min\set{\deg(x), \deg(y)}/2\rfloor.\]
Bonamy, Delcourt, Lang, and Postle provided an asymptotic bound in this setting.
They show that for any $\eps > 0$ and $\Delta$  sufficiently large, there is a proper $L$-edge-coloring whenever $|L(e)| \geq (1+\eps)\max\set{\deg(x), \deg(y)}$ for each $e \in E_G(x, y)$, provided $\Delta(G) \leq \Delta$ and $\delta(G) \geq \ln^{25}\Delta$ \cite{bonamy2024edge}.

Improving the results of \cite{bonamy2024edge} further seems a difficult task as it would result in strong advancements toward the list edge-coloring conjecture.
However, there has been success with certain types of lists, namely sets of the form $[r]$.
Recently, Christiansen \cite{Christ} resolved a conjecture posed in \cite{bonamy2024edge} by proving that all simple graphs admit an $L$-edge-coloring for $L(xy) = [1 + \max\set{\deg(x), \deg(y)}]$.
Conley, Greb{\'\i}k, and Pikhurko extended this result to multigraphs by constructing an $L$-edge-coloring for $L(e) = [\max\set{\deg(x) + \mu(x), \deg(y) + \mu(y)}]$, where $V(e) = \set{x, y}$ and $\mu(x) \defeq \max_{z}\mu(x, z)$ \cite{localvizing}.

We introduce a slightly different definition of locality, which generalizes these results.
In particular, given a multigraph $G$ and a list assignment $L$ for the edges of $G$, we define a function on the vertices of $G$ as follows:
\begin{align}\label{eqn:local_list_size}
    f(x, L) \defeq \bigcap_{e \in E_G(x)}L(e), \quad x \in V(G).
\end{align}
We will provide sufficient lower bounds on $|f(x, L)|$ in terms of the local structure of $x$ to guarantee the existence of a proper $L$-edge-coloring.
We note that this is a natural abstraction of local edge-colorings considered in \cite{Christ, localvizing} as the constraint $|f(x, L)| \geq c(x)$ can be achieved by the list assignment $L(e) = [\max\set{c(x), c(y)}]$ (where $V(e) = \set{x, y}$).
Additionally, our proofs are based on the ideas of \cite{localvizing} inspired by that of \cite{Christ}.

We are now ready to state our main result.

\begin{theo}\label{theo:main_theo}
    Let $G$ be an $n$-vertex multigraph of maximum degree $\Delta$ and let $L \,:\,E(G) \to 2^{\N}$ be a list assignment for the edges of $G$.
    Furthermore, let $f(x, L)$ be as defined in \eqref{eqn:local_list_size}.
    \begin{enumerate}
        \item\label{theo:shannon} If $|f(x, L)| \geq \Shav{x}$ for each $x \in V(G)$, then $G$ contains a proper $L$-edge-coloring.

        \item\label{theo:vizing} If $|f(x, L)| \geq \deg(x) + \mu(x)$ for each $x \in V(G)$, then $G$ contains a proper $L$-edge-coloring.

        \item\label{theo:vizing_bipartite} If $G$ is bipartite and $|f(x, L)| \geq \deg(x)$ for each $x \in V(G)$, then $G$ contains a proper $L$-edge-coloring.
    \end{enumerate}
    Moreover, if $|f(x, L)| \leq \poly(\Delta, n)$ for each $x$, there is a $\poly(\Delta, n)$ time algorithm to compute such colorings.
\end{theo}

We remark that the sets $f(x, L)$ would be included in the input of any such algorithm.
The following proposition allows us to assume $\max_x|f(x, L)| < \infty$ for the setting of Theorem~\ref{theo:main_theo}.

\begin{prop}\label{prop:list_size}
    Let $G$ be an $n$-vertex multigraph of maximum degree $\Delta$ and let $L \,:\,E(G) \to 2^{\N}$ be a list assignment for the edges of $G$.
    Let $c \,:\, V(G) \to \N$ be such that
    \begin{enumerate}
        \item\label{item:finite_c} there is $\ell \in \N$ such that $c(x) \leq \ell$ for each $x \in V(G)$, and
        \item\label{item:list_lb} $|f(x, L)| \geq c(x)$ for each $x \in V(G)$.
    \end{enumerate}
    Then there exists a list assignment for the edges $L'\,:\,E(G) \to 2^\N$ such that
    \begin{enumerate}
        \setcounter{enumi}{2}
        \item\label{item:subset} $L'(e) \subseteq L(e)$ for each $e \in E(G)$,
        \item\label{item:bdd_above} $|f(x, L')| < \infty$ for each $x \in V(G)$, and
        \item\label{item:bdd_below} $|f(x, L')| \geq c(x)$ for each $x \in V(G)$.
    \end{enumerate}
\end{prop}

\begin{proof}
    Define the list assignment $L_v\,:\, V(G) \to 2^\N$ such that $L_v(x) = f(x ,L)$.
    Let $L_v'(x)$ be an arbitrary subset of $L_v(x)$ containing $\min\set{|L_v(x)|, \ell}$ elements.
    In particular, if $|L_v(x)| < \ell$ then $L_v'(x) = L_v(x)$.
    Define $L'\,:\, E(G) \to 2^\N$ such that $L'(e) = L_v'(x) \cup L_v'(y)$ for each $e \in E_G(x, y)$.
    Note that as $L_v'(x) \subseteq f(x, L)$, it follows by \eqref{eqn:local_list_size} that \ref{item:subset} holds.
    Furthermore, we note that for any vertex $x \in V(G)$ and edge $e \in E_G(x)$
    \[f(x, L') \subseteq L'(e) \implies |f(x, L')| \leq 2\,\ell < \infty,\]
    implying \ref{item:bdd_above} holds.
    By construction, we also have $L_v'(x) \subseteq f(x, L')$.
    If $|L_v(x)| \geq \ell$, then we have
    \[|f(x, L')| \geq |L_v'(x)| = \ell \geq c(x),\]
    where the last step follows by \ref{item:finite_c}.
    If $|L_v(x)| < \ell$, then we have
    \[|f(x, L')| \geq |L_v'(x)| = |L_v(x)| \geq c(x),\]
    where the last step follows by \ref{item:list_lb}.
    In either case \ref{item:bdd_below} holds, completing the proof.
\end{proof}

For appropriate $\ell \in \set{\Sha, \Delta+\mu, \Delta}$, we may apply Proposition~\ref{prop:list_size} to assume $\max_x|f(x, L)| < \infty$ in Theorem~\ref{theo:main_theo}\eqref{theo:shannon}--\eqref{theo:vizing_bipartite}.
For the remainder of the paper, we will only consider list assignments satisfying $\max_x|f(x, L)| < \infty$.

Let $e \in E_G(x, y)$.
We remark that by setting $L(e) = [\max\set{\deg(x) + \mu(x), \deg(y) + \mu(y)}]$, we recover the result of \cite{localvizing}.
Furthermore, by setting $L(e) = [\max\set{\deg(x), \deg(y)}]$, we recover the local K\H{o}nig theorem.
Finally, by setting $L(e) = [\max\set{\Shav{x}, \Shav{y}}]$, we prove an analogous local result for Shannon's bound, which is slightly weaker than the one implied by the results in \cite{borodin1997list}.

\begin{corl}[to Theorem~\ref{theo:main_theo}]\label{theo:local_bounds}
    Let $G$ be an $n$-vertex multigraph of maximum degree $\Delta$.
    \begin{enumerate}
        \item There is a proper edge-coloring $\phi\,:\,E(G) \to \N$ such that for each $e \in E(G)$, we have $\phi(e) \in [\max\set{\Shav{x}, \Shav{y}}]$, where $V(e) = \set{x, y}$.

        \item There is a proper edge-coloring $\phi\,:\,E(G) \to \N$ such that for each $e \in E(G)$, we have $\phi(e) \in [\max\set{\deg(x) + \mu(x), \deg(y) + \mu(y)}]$, where $V(e) = \set{x, y}$.

        \item If $G$ is bipartite, there is a proper edge-coloring $\phi\,:\,E(G) \to \N$ such that for each $e \in E(G)$, we have $\phi(e) \in [\max\set{\deg(x), \deg(y)}]$, where $V(e) = \set{x, y}$.
    \end{enumerate}
    Moreover, there is a $\poly(\Delta, n)$ time algorithm to compute such colorings.
\end{corl}

The remainder of the paper is structured as follows. 
In \S\ref{section:prelim}, we will introduce the relevant notation and terminology we use as well as prove an important technical lemma which implies Theorem~\ref{theo:main_theo}\eqref{theo:vizing_bipartite}.
In \S\ref{section:shannon}, we will prove Theorem~\ref{theo:main_theo}\eqref{theo:shannon}, and in \S\ref{section:vizing}, we will prove Theorem~\ref{theo:main_theo}\eqref{theo:vizing}.

\section{Notation and Preliminaries}\label{section:prelim}

Throughout the remainder of the paper, $G$ is an $n$-vertex multigraph of maximum degree $\Delta$, $L\,:\, E(G) \to 2^{\N}$ is a list assignment for the edges of $G$, and $f(x, L)$ is as defined in \eqref{eqn:local_list_size}.
Let $\phi\,:\, E(G)\to \N \cup \{\blank\}$ be such that $\phi(e) \in L(e) \cup \set{\blank}$.
We call such a function a \emphd{partial $L$-edge-coloring} of $G$. 
Here, $\phi(e) = \blank$ indicates that the edge $e$ is uncolored.
We let $\mathsf{U_\phi}$ denote the set of uncolored edges under $\phi$.
Given a partial $L$-edge-coloring $\phi$ and $x\in V(G)$, we let
\[U(\phi, x) \defeq \{\phi(e)\,:\, e \in E_G(x)\setminus \mathsf U_\phi\}, \quad A(\phi, x) \defeq f(x, L) \setminus U(\phi, x)\] 
be the set of all the \emphd{used} and \emphd{available} colors at $x$ under the coloring $\phi$, respectively.
We say a color is \emphd{missing} at $x$ if it is not contained in $U(\phi, x)$.
As locality is understood, for the remainder of the paper we will use partial coloring to mean partial $L$-edge-coloring for brevity.

We say an uncolored edge $e \in E_G(x,y)$ is \emphd{$\phi$-happy} if $A(\phi, x)\setminus U(\phi, y) \neq \0$ or $A(\phi, y) \setminus U(\phi, x) \neq \0$.
If $e \in E_G(x, y)$ is $\phi$-happy, we can extend the coloring $\phi$ by assigning any color in $A(\phi, x)\setminus U(\phi, y)$ or $A(\phi, y) \setminus U(\phi, x)$ to $e$ to obtain a new proper partial coloring.

Given a proper partial coloring, we wish to modify it in order to create a partial coloring with a happy edge.
To do so, we will construct so-called \emphd{augmenting chains}.
The remainder of this section is similar to \cite[\S2]{ShannonChain} with a few minor additions due to the introduction of locality.
A \emphd{chain} of length $k$ is a sequence of distinct edges $C = (e_0, \ldots, e_{k-1})$ such that $|V(e_i)\cap V(e_{i+1})| = 1$ for each $0 \leq i < k-1$. 
Let $\Start(C) \defeq e_0$ and $\End(C) \defeq e_{k-1}$ denote the first and the last edges of $C$ respectively and let $\length(C) \defeq k$ be the length of $C$.
To assist with defining how we modify partial colorings, we will define the $\Shift$ operation. 
Given a chain $C = (e_0, \ldots, e_{k-1})$, we define the coloring $\Shift(\phi, C)$ as follows:
\[\Shift(\phi, C)(e) \defeq \left\{\begin{array}{cc}
   \phi(e_{i+1})  & e = e_i, \, 0 \leq i < k-1; \\
    \blank & e = e_{k-1}; \\
    \phi(e) & \text{otherwise.}
\end{array}\right.\]
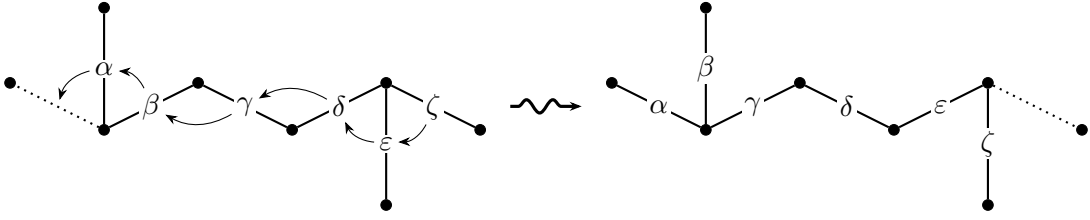
\begin{figure}[t]
	\centering
	\begin{tikzpicture}
	\begin{scope}
	\node[circle,fill=black,draw,inner sep=0pt,minimum size=4pt] (a) at (0,0) {};
	\node[circle,fill=black,draw,inner sep=0pt,minimum size=4pt] (b) at (-1.25,0.625) {};
	\node[circle,fill=black,draw,inner sep=0pt,minimum size=4pt] (c) at (0,1.625) {};
	\node[circle,fill=black,draw,inner sep=0pt,minimum size=4pt] (d) at (1.25,0.625) {};
	\node[circle,fill=black,draw,inner sep=0pt,minimum size=4pt] (e) at (2.5,0) {};
	\node[circle,fill=black,draw,inner sep=0pt,minimum size=4pt] (f) at (3.75,0.625) {};
	\node[circle,fill=black,draw,inner sep=0pt,minimum size=4pt] (g) at (3.75,-1) {};
	\node[circle,fill=black,draw,inner sep=0pt,minimum size=4pt] (h) at (5,0) {};
	
	\draw[thick,dotted] (a) to node[midway,inner sep=0pt,minimum size=4pt] (i) {} (b);
	\draw[thick] (a) to node[midway,inner sep=1pt,outer sep=1pt,minimum size=4pt,fill=white] (j) {$\alpha$} (c);
	\draw[thick] (a) to node[midway,inner sep=1pt,outer sep=1pt,minimum size=4pt,fill=white] (k) {$\beta$} (d);
	\draw[thick] (d) to node[midway,inner sep=1pt,outer sep=1pt,minimum size=4pt,fill=white] (l) {$\gamma$} (e);
	\draw[thick] (f) to node[midway,inner sep=1pt,outer sep=1pt,minimum size=4pt,fill=white] (m) {$\delta$} (e);
	\draw[thick] (f) to node[midway,inner sep=1pt,outer sep=1pt,minimum size=4pt,fill=white] (n) {$\epsilon$} (g);
	\draw[thick] (f) to node[midway,inner sep=1pt,outer sep=1pt,minimum size=4pt,fill=white] (o) {$\zeta$} (h);
	
	\draw[-{Stealth[length=1.6mm]}] (j) to[bend right] (i);
	\draw[-{Stealth[length=1.6mm]}] (k) to[bend right] (j);
	\draw[-{Stealth[length=1.6mm]}] (l) to[bend left] (k);
	\draw[-{Stealth[length=1.6mm]}] (m) to[bend right] (l);
	\draw[-{Stealth[length=1.6mm]}] (n) to[bend left] (m);
	\draw[-{Stealth[length=1.6mm]}] (o) to[bend left] (n);
	\end{scope}
	
	\draw[-{Stealth[length=1.6mm]},very thick,decoration = {snake,pre length=3pt,post length=7pt,},decorate] (5.4,0.3125) -- (6.35,0.3125);
	
	\begin{scope}[xshift=8cm]
	\node[circle,fill=black,draw,inner sep=0pt,minimum size=4pt] (a) at (0,0) {};
	\node[circle,fill=black,draw,inner sep=0pt,minimum size=4pt] (b) at (-1.25,0.625) {};
	\node[circle,fill=black,draw,inner sep=0pt,minimum size=4pt] (c) at (0,1.625) {};
	\node[circle,fill=black,draw,inner sep=0pt,minimum size=4pt] (d) at (1.25,0.625) {};
	\node[circle,fill=black,draw,inner sep=0pt,minimum size=4pt] (e) at (2.5,0) {};
	\node[circle,fill=black,draw,inner sep=0pt,minimum size=4pt] (f) at (3.75,0.625) {};
	\node[circle,fill=black,draw,inner sep=0pt,minimum size=4pt] (g) at (3.75,-1) {};
	\node[circle,fill=black,draw,inner sep=0pt,minimum size=4pt] (h) at (5,0) {};
	
	\draw[thick] (a) to node[midway,inner sep=1pt,outer sep=1pt,minimum size=4pt,fill=white] (i) {$\alpha$} (b);
	\draw[thick] (a) to node[midway,inner sep=1pt,outer sep=1pt,minimum size=4pt,fill=white] (j) {$\beta$} (c);
	\draw[thick] (a) to node[midway,inner sep=1pt,outer sep=1pt,minimum size=4pt,fill=white] (k) {$\gamma$} (d);
	\draw[thick] (d) to node[midway,inner sep=1pt,outer sep=1pt,minimum size=4pt,fill=white] (l) {$\delta$} (e);
	\draw[thick] (f) to node[midway,inner sep=1pt,outer sep=1pt,minimum size=4pt,fill=white] (m) {$\epsilon$} (e);
	\draw[thick] (f) to node[midway,inner sep=1pt,outer sep=1pt,minimum size=4pt,fill=white] (n) {$\zeta$} (g);
	\draw[thick,dotted] (f) to node[midway,inner sep=0pt,minimum size=4pt] (o) {} (h);
	\end{scope}
	\end{tikzpicture}
	\caption{Shifting a coloring along a chain (Greek letters represent colors).}\label{fig:shift}
\end{figure}

\noindent 
In other words, $\Shift(\phi, C)$ ``shifts'' the color from $e_{i+1}$ to $e_i$, leaves $e_{k-1}$ uncolored, and keeps the coloring on the other edges unchanged (see Fig.~\ref{fig:shift} for an example).
We call $C$ a \emphd{$\phi$-shiftable chain} if $\phi(e_0) = \blank$ and the coloring $\Shift(\phi, C)$ is a proper partial coloring.

The next definition captures the class of chains that can be used to create a happy edge:

\begin{defn}[Happy chains]
    We say that a chain $C$ is \emphd{$\phi$-happy} for a partial coloring $\phi$ if it is $\phi$-shiftable and the edge $\End(C)$ is $\Shift(\phi, C)$-happy.
\end{defn}

It may not always be possible to find a $\phi$-happy chain.
In the case that we cannot find a $\phi$-happy chain, we wish to modify our coloring so that we may be able to find one in the new coloring.
To this end, we define a \emphd{potential function} of a proper partial coloring $\phi$ (this approach was introduced in \cite{Christ} and further developed in \cite{localvizing}):
\begin{align*}
    A(G, \phi) &\defeq \sum_{x \in V(G)}|A(\phi, x)|, \quad D(G, \phi) \defeq \sum_{x \in V(G)}\deg(x)|E_G(x) \cap \mathsf{U_\phi}|, \\
    \Pot(G, \phi) &\defeq (A(G, \phi), D(G, \phi)).
\end{align*}
We consider the lexicographic ordering on $\Z^2$:
\[(z_1, z_2) < (w_1, w_2) \iff z_1 < w_1 \text{ or } z_1 = w_1, z_2 < w_2.\]
Let us now define our final class of chains.

\begin{defn}
    We say that a chain $C$ is \emphd{$\phi$-content} for a partial coloring $\phi$ if it is $\phi$-shiftable and $\Pot(G, \Shift(\phi, C)) < \Pot(G, \phi)$.
\end{defn}

Note that 
\begin{align}\label{eqn:pot_bdd}
    0 \leq A(G, \phi) \leq n\max_{x}|f(x, L)|, \quad 0 \leq D(G, \phi) \leq 2m^2 \leq \Delta^2n^2/2.
\end{align}
We will describe procedures that compute chains that are either $\phi$-happy or $\phi$-content for any partial coloring $\phi$.
It follows that after finding a finite number of content chains, we are guaranteed to find a happy one.
Furthermore, we will show that these procedures run in $\poly(\Delta, n)$ time. 
Assuming $\max_{x}|f(x, L)| \leq \poly(\Delta, n)$, it follows from \eqref{eqn:pot_bdd} that we can compute such colorings in $\poly(\Delta, n)$ time.

The first special type of chains we consider are path chains.

\begin{defn}[{\cite[Definition~2.2]{ShannonChain}}]\label{defn:path}
    A chain $P = (e_0, \ldots, e_{k-1})$ is a \emphd{path chain} if the edges $e_1$, \ldots, $e_{k-1}$ form a path in $G$, i.e., if there exist distinct vertices $x_1$, \ldots, $x_k$ such that $e_i \in E_G(x_i,x_{i+1})$ for all $1 \leq i \leq k-1$ (see Fig.~\ref{fig:path} for examples). We let $x_0 \in V(e_0)$ be distinct from $x_1$ and let $\vstart(P) \defeq x_0$, $\vend(P) \defeq x_k$ denote the first and last vertices on the path chain respectively. 
    Note that the vertex $x_0$ may coincide with $x_i$ for some $3 \leq i \leq k$; see Fig.~\ref{subfig:unsucc} for an example. 
    The vertices $\vstart(P)$ and $\vend(P)$ are uniquely determined unless $P$ is a single edge.
\end{defn}

\begin{figure}[t]
	\centering
        \begin{subfigure}[t]{\textwidth}
		\centering
		\begin{tikzpicture}
		\node[circle,fill=black,draw,inner sep=0pt,minimum size=4pt] (x) at (0,0) {};
            \node[circle,fill=black,draw,inner sep=0pt,minimum size=4pt] (y) at (1,0) {};
            \node[circle,fill=black,draw,inner sep=0pt,minimum size=4pt] (a) at (2,0) {};
            \node[circle,fill=black,draw,inner sep=0pt,minimum size=4pt] (b) at (3,0) {};
            \node[circle,fill=black,draw,inner sep=0pt,minimum size=4pt] (c) at (5,0) {};
            \node[circle,fill=black,draw,inner sep=0pt,minimum size=4pt] (d) at (6,0) {};
            \node[circle,fill=black,draw,inner sep=0pt,minimum size=4pt] (e) at (7,0) {};
            \node[circle,fill=black,draw,inner sep=0pt,minimum size=4pt] (f) at (-1,0) {};
            \node[circle,fill=black,draw,inner sep=0pt,minimum size=4pt] (g) at (-2,0) {};
            \node[circle,fill=black,draw,inner sep=0pt,minimum size=4pt] (h) at (-4,0) {};
            \node[circle,fill=black,draw,inner sep=0pt,minimum size=4pt] (i) at (-5,0) {};
            \node[circle,fill=black,draw,inner sep=0pt,minimum size=4pt] (j) at (-6,0) {};

		\draw[thick,dotted] (x) to node[midway,below,inner sep=1pt,outer sep=1pt,minimum size=4pt,fill=white] {$e$} (y);
		\draw[thick] (y) to node[midway,inner sep=1pt,outer sep=1pt,minimum size=4pt,fill=white] {$\alpha$} (a) to node[midway,inner sep=1pt,outer sep=1pt,minimum size=4pt,fill=white] {$\beta$} (b) (c) to node[midway,inner sep=1pt,outer sep=1pt,minimum size=4pt,fill=white] {$\beta$} (d) to node[midway,inner sep=1pt,outer sep=1pt,minimum size=4pt,fill=white] {$\alpha$} (e) 
        (x) to node[midway,inner sep=1pt,outer sep=1pt,minimum size=4pt,fill=white] {$\beta$} (f) to node[midway,inner sep=1pt,outer sep=1pt,minimum size=4pt,fill=white] {$\alpha$} (g) (h) to node[midway,inner sep=1pt,outer sep=1pt,minimum size=4pt,fill=white] {$\alpha$} (i) to node[midway,inner sep=1pt,outer sep=1pt,minimum size=4pt,fill=white] {$\beta$} (j);
        \draw[thick, snake=zigzag] (b) -- (c) (g) -- (h);

        \draw[decoration={brace,amplitude=10pt},decorate] (-6,0.2) -- node [midway,above,yshift=10pt,xshift=0pt] {$P(e; \phi, \beta\alpha)$} (1,0.2);
        \draw[decoration={brace,amplitude=10pt,mirror},decorate] (0,-0.2) -- node [midway,above,yshift=-28pt,xshift=0pt] {$P(e; \phi, \alpha\beta)$} (7,-0.2);
		
		\end{tikzpicture}
		\caption{$P(e; \phi, \alpha\beta)$ and $P(e; \phi, \beta\alpha)$.}\label{subfig:diff_paths}
	\end{subfigure}%
        \vspace{10pt}
	\begin{subfigure}[t]{.47\textwidth}
		\centering
		\begin{tikzpicture}
		\node[draw=none,minimum size=2.5cm,regular polygon,regular polygon sides=7] (P) {};

		\node[circle,fill=black,draw,inner sep=0pt,minimum size=4pt] (x) at (P.corner 4) {};
		\node[circle,fill=black,draw,inner sep=0pt,minimum size=4pt] (y) at (P.corner 5) {};
		\node[circle,fill=black,draw,inner sep=0pt,minimum size=4pt] (a) at (P.corner 6) {};
		\node[circle,fill=black,draw,inner sep=0pt,minimum size=4pt] (b) at (P.corner 7) {};
		\node[circle,fill=black,draw,inner sep=0pt,minimum size=4pt] (c) at (P.corner 1) {};
		\node[circle,fill=black,draw,inner sep=0pt,minimum size=4pt] (d) at (P.corner 2) {};
		\node[circle,fill=black,draw,inner sep=0pt,minimum size=4pt] (e) at (P.corner 3) {};
		
		\node[anchor=north] at (x) {$x$};
		\node[anchor=north] at (y) {$y$};
		
		\draw[thick,dotted] (x) to node[midway,below,inner sep=1pt,outer sep=1pt,minimum size=4pt,fill=white] {$e$} (y);
		\draw[thick] (y) to node[midway,inner sep=1pt,outer sep=1pt,minimum size=4pt,fill=white] {$\alpha$} (a);
		\draw[thick] (a) to node[midway,inner sep=1pt,outer sep=1pt,minimum size=4pt,fill=white] {$\beta$} (b);
		\draw[thick] (b) to node[midway,inner sep=1pt,outer sep=1pt,minimum size=4pt,fill=white] {$\alpha$} (c);
		\draw[thick] (c) to node[midway,inner sep=1pt,outer sep=1pt,minimum size=4pt,fill=white] {$\beta$} (d);
		\draw[thick] (d) to node[midway,inner sep=1pt,outer sep=1pt,minimum size=4pt,fill=white] {$\alpha$} (e);
		\draw[thick] (e) to node[midway,inner sep=1pt,outer sep=1pt,minimum size=4pt,fill=white] {$\beta$} (x);
		\end{tikzpicture}
		\caption{A path chain where $\vstart(P) = \vend(P)$.}\label{subfig:unsucc}
	\end{subfigure}%
	\qquad%
	\begin{subfigure}[t]{.47\textwidth}
		\centering
		\begin{tikzpicture}
		\node[draw=none,minimum size=2.5cm,regular polygon,regular polygon sides=7] (P) {};

		\node[circle,fill=black,draw,inner sep=0pt,minimum size=4pt] (x) at (P.corner 4) {};
		\node[circle,fill=black,draw,inner sep=0pt,minimum size=4pt] (y) at (P.corner 5) {};
		\node[circle,fill=black,draw,inner sep=0pt,minimum size=4pt] (a) at (P.corner 6) {};
		\node[circle,fill=black,draw,inner sep=0pt,minimum size=4pt] (b) at (P.corner 7) {};
		\node[circle,fill=black,draw,inner sep=0pt,minimum size=4pt] (c) at (P.corner 1) {};
		\node[circle,fill=black,draw,inner sep=0pt,minimum size=4pt] (d) at (P.corner 2) {};
		\node[circle,fill=black,draw,inner sep=0pt,minimum size=4pt] (e) at (P.corner 3) {};
		\node[circle,fill=black,draw,inner sep=0pt,minimum size=4pt] (f) at (-2.2,0) {}; 
		
		\node[anchor=north] at (x) {$x$};
		\node[anchor=north] at (y) {$y$};
		
		\draw[thick,dotted] (x) to node[midway,below,inner sep=1pt,outer sep=1pt,minimum size=4pt,fill=white] {$e$} (y);
		\draw[thick] (y) to node[midway,inner sep=1pt,outer sep=1pt,minimum size=4pt,fill=white] {$\alpha$} (a);
		\draw[thick] (a) to node[midway,inner sep=1pt,outer sep=1pt,minimum size=4pt,fill=white] {$\beta$} (b);
		\draw[thick] (b) to node[midway,inner sep=1pt,outer sep=1pt,minimum size=4pt,fill=white] {$\alpha$} (c);
		\draw[thick] (c) to node[midway,inner sep=1pt,outer sep=1pt,minimum size=4pt,fill=white] {$\beta$} (d);
		\draw[thick] (d) to node[midway,inner sep=1pt,outer sep=1pt,minimum size=4pt,fill=white] {$\alpha$} (e);
		\draw[thick] (e) to node[midway,inner sep=1pt,outer sep=1pt,minimum size=4pt,fill=white] {$\beta$} (f);
		\end{tikzpicture}
		\caption{A path chain where $\vstart(P) \neq \vend(P)$.}
	\end{subfigure}
	\caption{Path chains.}
        \label{fig:path}
\end{figure}
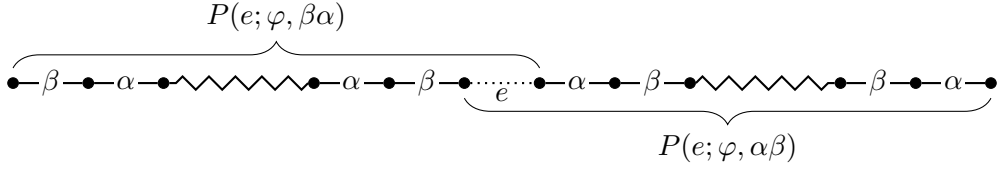
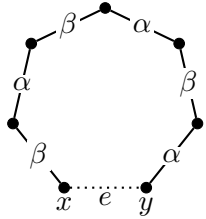
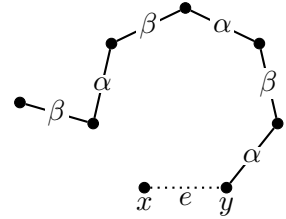

We are specifically interested in path chains containing at most $2$ colors, except at their first edge.
For $\alpha, \beta \in \N$, we say that a path chain $P = (e_0, \ldots, e_{k-1})$ is an \emphd{$\alpha\beta$-path} under $\phi$ if the sequence $(\phi(e_1), \ldots, \phi(e_{k-1}))$ alternates between $\alpha$ and $\beta$.
For a partial coloring $\phi$, an uncolored edge $e \in E_G(x, y)$ and colors $\alpha \in A(\phi, x),\, \beta \in A(\phi, y)$, we define the $\alpha\beta$-path chain  $P(e; \phi, \alpha\beta) \defeq (e_0 = e, e_1, \ldots, e_{k-1})$ to be a maximal $\alpha\beta$-path starting at $e$ such that $\phi(e_1) = \alpha$ (see Fig.~\ref{subfig:diff_paths}).

For a chain $C = (e_0, \ldots, e_{k-1})$ and $1 \leq j \leq k$, we define the \emphd{initial segment} of $C$ as
\[C|j \,\defeq\, (e_0, \ldots, e_{j-1}).\]
The following lemma will be crucial for the rest of the proof.

\begin{Lemma}\label{lemma:path}
    Let $\phi$ be a proper partial $L$-edge-coloring and let $e \in E_G(x, y)$ be an uncolored edge.
    Let $P \defeq P(e; \phi, \alpha, \beta)$ be such that:
    \begin{enumerate}
        \item\label{item:alpha} $\alpha \in A(\phi, x)$, 
        \item\label{item:beta} $\beta \in A(\phi, y)$, and
        \item\label{item:successful} $\vstart(P) \neq \vend(P)$.
    \end{enumerate}
    Then either $P$ is $\phi$-happy or some initial segment $P'$ of $P$ satisfies $A(G, \Shift(\phi, P')) < A(G, \phi)$.
\end{Lemma}

\begin{proof}
    Note that $P$ is maximal, i.e., either $\alpha$ or $\beta$ is missing at $\vend(P)$ under $\phi$.
    
    First, we claim that if $\length(P) < 3$, then $P$ is $\phi$-happy.
    If $P = (e)$, $\alpha$ is missing at $y$. 
    By \ref{item:alpha}, it follows that $e$ is $\phi$-happy.
    Now suppose $P = (e, f)$ such that $\phi(f) = \alpha$. 
    Let $z \in V(f)$ be distinct from $y$.
    By the definition of a chain $z \neq x$.
    By \ref{item:alpha}, $P$ is $\phi$-shiftable.
    We have:
    \[A(\Shift(\phi, P), y) = A(\phi, y), \quad U(\Shift(\phi, P), z) = U(\phi, z) \setminus \set{\alpha}.\]
    As $P$ is maximal, we know that $\beta \notin U(\phi, z)$ and so $\beta \notin U(\Shift(\phi, P), z)$ as well.
    Therefore, $P$ is $\phi$-happy, as claimed.

    Let us now assume $\length(P) \geq 3$ and let $1 \leq j \leq \length(P)$ be the largest integer such that $\psi \defeq \Shift(\phi, P|j)$ is a proper $L$-edge-coloring.
    By \ref{item:alpha} and \ref{item:beta}, we have $j \geq 3$.
    Let $f \defeq \End(P|j)$ and $V(f) = \set{u, v}$ such that $v = \vend(P|j)$ (see Fig.~\ref{fig:shift_path}).
    Since $j \geq 3$, by Definition~\ref{defn:path} it follows that $y \notin V(f)$.
    We note that $\set{\alpha, \beta} \subseteq U(\phi, u)$.
    Furthermore, $\set{\alpha, \beta} \subseteq U(\phi, v)$ unless $v = \vend(P)$.
    It follows by \ref{item:alpha} and \ref{item:successful} that $x \notin V(f)$ as well.
    For each $z \in V(P|j) \setminus \set{u, v, x, y}$, we have $A(\psi, z) = A(\phi, z)$ as the only changes in their neighborhoods are that the edges colored $\alpha$ and $\beta$ swap values.
    Furthermore, as $j \geq 3$, we have
    \begin{align}\label{eqn:avail_sets}
        A(\psi, x) = A(\phi, x) \setminus\set{\alpha}, \quad A(\psi, y) = A(\phi, y) \setminus\set{\beta}.
    \end{align}
    Without loss of generality, let us assume $\phi(f) = \beta$.
    By construction, we have
    \[U(\psi, u) = U(\phi, u) \setminus \set{\alpha}, \quad U(\psi, v) = U(\phi, v) \setminus \set{\beta}.\]
    
    If $j = k$ and $\alpha \in L(f)$, then either $P$ is $\phi$-happy and we are done or $\alpha \notin A(\psi, u) \cup A(\psi, v)$.
    It follows that $A(\psi, u) = A(\phi, u)$ and $A(\psi, v)$ contains at most $1$ more color than $A(\phi, v)$, namely $\beta$.
    
    If $j < k$, by our choice of $j$, we must have
    \[\alpha \notin L(f) \implies \alpha \notin f(u, L) \cup f(v, L).\]
    In particular, $A(\psi, u) = A(\phi, u)$ and $A(\psi, v)$ contains at most $1$ more color than $A(\phi, v)$, namely $\beta$.
    
    Both of the above cases along with \eqref{eqn:avail_sets} imply that $A(G, \psi) \leq A(G, \phi) - 1$, as desired.
\end{proof}

\begin{figure}[t]
    \centering
        \begin{tikzpicture}
            \node[circle,fill=black,draw,inner sep=0pt,minimum size=4pt] (a) at (0,0) {};
            \path (a) ++(0:1) node[circle,fill=black,draw,inner sep=0pt,minimum size=4pt] (b) {};
            \path (b) ++(0:1) node[circle,fill=black,draw,inner sep=0pt,minimum size=4pt] (e) {};

            \path (e) ++(0:1) node[circle,fill=black,draw,inner sep=0pt,minimum size=4pt] (f) {};
            \path (f) ++(0:1) node[circle,fill=black,draw,inner sep=0pt,minimum size=4pt] (g) {};
            \path (g) ++(0:1) node[circle,fill=black,draw,inner sep=0pt,minimum size=4pt] (h) {};
            \path (h) ++(0:1) node[circle,fill=black,draw,inner sep=0pt,minimum size=4pt] (i) {};
            \path (i) ++(0:1) node[circle,fill=black,draw,inner sep=0pt,minimum size=4pt] (j) {};
            \path (j) ++(0:1) node[circle,fill=black,draw,inner sep=0pt,minimum size=4pt] (k) {};
            \path (k) ++(0:1) node[circle,fill=black,draw,inner sep=0pt,minimum size=4pt] (l) {};
            \path (l) ++(0:1) node[circle,fill=black,draw,inner sep=0pt,minimum size=4pt] (m) {};
            \path (m) ++(0:1) node[circle,fill=black,draw,inner sep=0pt,minimum size=4pt] (n) {};
            \path (n) ++(0:1) node[circle,fill=black,draw,inner sep=0pt,minimum size=4pt] (o) {};

            \node[anchor=north] at (k) {$u$};
            \node[anchor=north] at (l) {$v$};
            \node[anchor=north] at (a) {$x$};
            \node[anchor=north] at (b) {$y$};

            \draw[thick] (b) to node[font=\fontsize{8}{8},midway,inner sep=1pt,outer sep=1pt,minimum size=4pt,fill=white] {$\alpha$} (e) to node[font=\fontsize{8}{8},midway,inner sep=1pt,outer sep=1pt,minimum size=4pt,fill=white] {$\beta$} (f) to node[font=\fontsize{8}{8},midway,inner sep=1pt,outer sep=1pt,minimum size=4pt,fill=white] {$\alpha$} (g) to node[font=\fontsize{8}{8},midway,inner sep=1pt,outer sep=1pt,minimum size=4pt,fill=white] {$\beta$} (h) to node[font=\fontsize{8}{8},midway,inner sep=1pt,outer sep=1pt,minimum size=4pt,fill=white] {$\alpha$} (i) to node[font=\fontsize{8}{8},midway,inner sep=1pt,outer sep=1pt,minimum size=4pt,fill=white] {$\beta$} (j) to node[font=\fontsize{8}{8},midway,inner sep=1pt,outer sep=1pt,minimum size=4pt,fill=white] {$\alpha$} (k) to node[font=\fontsize{8}{8},midway,inner sep=1pt,outer sep=1pt,minimum size=4pt,fill=white] {$\beta$} (l) to node[font=\fontsize{8}{8},midway,inner sep=1pt,outer sep=1pt,minimum size=4pt,fill=white] {$\alpha$} (m) to node[font=\fontsize{8}{8},midway,inner sep=1pt,outer sep=1pt,minimum size=4pt,fill=white] {$\beta$} (n) to node[font=\fontsize{8}{8},midway,inner sep=1pt,outer sep=1pt,minimum size=4pt,fill=white] {$\alpha$} (o);

            \draw[thick, dotted] (a) -- (b);

            \begin{scope}[yshift=-1.5cm]
                \draw[-{Stealth[length=3mm,width=2mm]},very thick,decoration = {snake,pre length=3pt,post length=7pt,},decorate] (6,1) -- (6,0);
            \end{scope}

            \begin{scope}[yshift=-2cm]
                \node[circle,fill=black,draw,inner sep=0pt,minimum size=4pt] (a) at (0,0) {};
                \path (a) ++(0:1) node[circle,fill=black,draw,inner sep=0pt,minimum size=4pt] (b) {};
                \path (b) ++(0:1) node[circle,fill=black,draw,inner sep=0pt,minimum size=4pt] (e) {};
    
                \path (e) ++(0:1) node[circle,fill=black,draw,inner sep=0pt,minimum size=4pt] (f) {};
                \path (f) ++(0:1) node[circle,fill=black,draw,inner sep=0pt,minimum size=4pt] (g) {};
                \path (g) ++(0:1) node[circle,fill=black,draw,inner sep=0pt,minimum size=4pt] (h) {};
                \path (h) ++(0:1) node[circle,fill=black,draw,inner sep=0pt,minimum size=4pt] (i) {};
                \path (i) ++(0:1) node[circle,fill=black,draw,inner sep=0pt,minimum size=4pt] (j) {};
                \path (j) ++(0:1) node[circle,fill=black,draw,inner sep=0pt,minimum size=4pt] (k) {};
                \path (k) ++(0:1) node[circle,fill=black,draw,inner sep=0pt,minimum size=4pt] (l) {};
                \path (l) ++(0:1) node[circle,fill=black,draw,inner sep=0pt,minimum size=4pt] (m) {};
                \path (m) ++(0:1) node[circle,fill=black,draw,inner sep=0pt,minimum size=4pt] (n) {};
                \path (n) ++(0:1) node[circle,fill=black,draw,inner sep=0pt,minimum size=4pt] (o) {};
    
                \node[anchor=north] at (k) {$u$};
                \node[anchor=north] at (l) {$v$};
                \node[anchor=north] at (a) {$x$};
                \node[anchor=north] at (b) {$y$};
    
                \draw[thick] (a) to node[font=\fontsize{8}{8},midway,inner sep=1pt,outer sep=1pt,minimum size=4pt,fill=white] {$\alpha$} (b) to node[font=\fontsize{8}{8},midway,inner sep=1pt,outer sep=1pt,minimum size=4pt,fill=white] {$\beta$} (e) to node[font=\fontsize{8}{8},midway,inner sep=1pt,outer sep=1pt,minimum size=4pt,fill=white] {$\alpha$} (f) to node[font=\fontsize{8}{8},midway,inner sep=1pt,outer sep=1pt,minimum size=4pt,fill=white] {$\beta$} (g) to node[font=\fontsize{8}{8},midway,inner sep=1pt,outer sep=1pt,minimum size=4pt,fill=white] {$\alpha$} (h) to node[font=\fontsize{8}{8},midway,inner sep=1pt,outer sep=1pt,minimum size=4pt,fill=white] {$\beta$} (i) to node[font=\fontsize{8}{8},midway,inner sep=1pt,outer sep=1pt,minimum size=4pt,fill=white] {$\alpha$} (j) to node[font=\fontsize{8}{8},midway,inner sep=1pt,outer sep=1pt,minimum size=4pt,fill=white] {$\beta$} (k) (l) to node[font=\fontsize{8}{8},midway,inner sep=1pt,outer sep=1pt,minimum size=4pt,fill=white] {$\alpha$} (m) to node[font=\fontsize{8}{8},midway,inner sep=1pt,outer sep=1pt,minimum size=4pt,fill=white] {$\beta$} (n) to node[font=\fontsize{8}{8},midway,inner sep=1pt,outer sep=1pt,minimum size=4pt,fill=white] {$\alpha$} (o);
    
                \draw[thick, dotted] (k) -- (l);
            \end{scope}
        \end{tikzpicture}    
    \caption{Shifting an initial segment of $P(e; \phi, \alpha\beta)$ for $e \in E_G(x, y)$.}
    \label{fig:shift_path}
\end{figure}
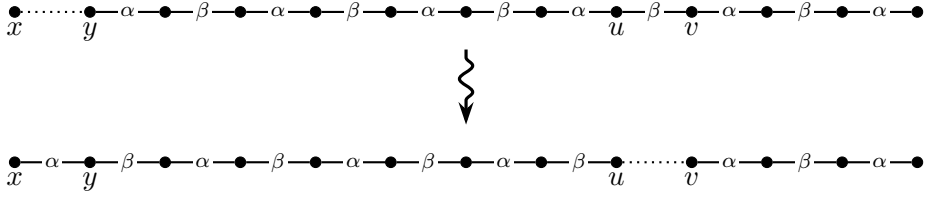

In order to prove Theorem~\ref{theo:main_theo}, we will describe a procedure to find a chain $C$ such that:
\begin{itemize}
    \item either $C$ is $\phi$-happy, or
    \item $C$ is $\phi$-content, or
    \item $A(G, \phi) = A(G, \Shift(\phi, C))$ and $\Shift(\phi, C)$ contains a path chain $P$ with $\Start(P) = \End(C)$ satisfying the conditions of Lemma~\ref{lemma:path}.
\end{itemize}
Let us show that this is enough.
We consider the following procedure starting with a blank coloring $\phi$:
\begin{enumerate}[label=\ep{\normalfont{}\textbf{Step\arabic*}},labelindent=0pt,leftmargin=*]
    \item Compute the chain $C$ described earlier.

    \item\label{item:happy} If $C$ is $\phi$-happy, then update $\phi$ by coloring $\End(C)$ with a valid color in $\Shift(\phi, C)$.

    \item\label{item:content} If $C$ is $\phi$-content, then update $\phi$ to $\Shift(\phi, C)$.

    \item\label{item:path_happy} If $\Shift(\phi, C)$ contains a $\Shift(\phi, C)$-happy path $P$, then update $\phi$ by coloring $\End(P)$ with a valid color in $\Shift(\Shift(\phi, C), P)$.

    \item\label{item:path_content} If $\Shift(\phi, C)$ contains a $\Shift(\phi, C)$-content path $P$, then update $\phi$ to $\Shift(\Shift(\phi, C), P)$.
    
\end{enumerate}
These steps cover all cases.
As a result of \eqref{eqn:pot_bdd}, we reach \ref{item:content} or \ref{item:path_content} a finite number of times before reaching \ref{item:happy} or \ref{item:path_happy}.
We note that updating the coloring given $C,\,P$ takes at most $O(\Delta(\length(C)+ \length(P)))$ time.
The chain $C$ we will compute will have length at most $\Delta$, and it is easy to see that $\length(P) \leq n$.
We remark that given $\phi,\,e,\,\alpha,\,\beta$, we may compute $P(e;\phi, \alpha\beta)$ in $O(\Delta\,n)$ time.
We will also show that, provided $\max_x|f(x, L)| \leq \poly(\Delta, n)$, computing $C$ takes $\poly(\Delta, n)$ time and so it follows that this procedure computes a proper $L$-edge-coloring in $\poly(\Delta, n)$ time.

From standard proofs of the chromatic index of bipartite multigraphs, one can show that given any bipartite multigraph $G$ and partial coloring $\phi$, we can find a path chain $P$ satisfying the conditions of Lemma~\ref{lemma:path} provided $|f(x, L)| \geq \deg(x)$ (see \cite[Proposition~5.3.1]{Diestel} for an outline of such a proof).
Therefore, Theorem~\ref{theo:main_theo}\eqref{theo:vizing_bipartite} follows from Lemma~\ref{lemma:path}.

The chains $C$ we will compute for the proofs of Theorem~\ref{theo:main_theo}\eqref{theo:shannon}-\eqref{theo:vizing} will be fan chains.

\begin{defn}[{\cite[Definition~2.5]{ShannonChain}}]\label{defn:fan}
    A \emphd{fan} is a chain of the form $F = (e_0, \ldots, e_{k-1})$ such that all edges contain a common vertex $x$, referred to as the \emphd{pivot} of the fan (see Fig.~\ref{fig:fan} for an example).
    With $y_i$ defined as the vertex in $V(e_i)$ distinct from $x$,
    we let $\Pivot(F) \defeq x$, $\vstart(F) \defeq y_0$ and $\vend(F) \defeq y_{k-1}$ denote the pivot, start, and end vertices of a fan $F$. 
    This notation is uniquely determined unless $F$ is a single edge, in which case we let $\vend(F) = \vstart(F)$.
    We note that it may be the case that $y_i = y_j$ for $i < j$ as $G$ is a multigraph.
\end{defn}

\begin{figure}[t]
	\centering
	\begin{tikzpicture}[xscale=1.1]
	\begin{scope}
	\node[circle,fill=black,draw,inner sep=0pt,minimum size=4pt] (x) at (0,0) {};
	\node[anchor=north] at (x) {$x$};
	
	\coordinate (O) at (0,0);
	\def\radius{2.6cm}
	
	\node[circle,fill=black,draw,inner sep=0pt,minimum size=4pt] (y0) at (190:\radius) {};
	\node at (190:2.9) {$y_0$};
	
	\node[circle,fill=black,draw,inner sep=0pt,minimum size=4pt] (y1) at (165:\radius) {};
	\node at (165:2.9) {$y_1$};
	
	\node[circle,fill=black,draw,inner sep=0pt,minimum size=4pt] (y2) at (140:\radius) {};
	\node at (140:2.9) {$y_2$};
	
	\node[circle,fill=black,draw,inner sep=0pt,minimum size=4pt] (y4) at (90:\radius) {};
	\node at (90:2.9) {$y_{i-1}$};
	
	\node[circle,fill=black,draw,inner sep=0pt,minimum size=4pt] (y5) at (62:\radius) {};
	\node at (62:2.9) {$y_i = y_j$};
	
	\node[circle,fill=black,draw,inner sep=0pt,minimum size=4pt] (y6) at (30:\radius) {};
	\node at (30:3) {$y_{i+1}$};
	
	\node[circle,fill=black,draw,inner sep=0pt,minimum size=4pt] (y8) at (-10:\radius) {};
	\node at (-10:3.1) {$y_{k-1}$};
	
	\node[circle,inner sep=0pt,minimum size=4pt] at (115:2) {$\ldots$}; 
	\node[circle,inner sep=0pt,minimum size=4pt] at (12:2) {$\ldots$}; 
	
	\draw[thick,dotted] (x) to (y0);
	\draw[thick] (x) to node[midway,inner sep=1pt,outer sep=1pt,minimum size=4pt,fill=white] {$\beta_0$} (y1);
	\draw[thick] (x) to node[midway,inner sep=1pt,outer sep=1pt,minimum size=4pt,fill=white] {$\beta_1$} (y2);
	
	\draw[thick] (x) to node[pos=0.55,inner sep=1pt,outer sep=1pt,minimum size=4pt,fill=white] {$\beta_{i-2}$} (y4);
	\draw[thick] (x) to[out=80, in=-140] node[pos=0.75,inner sep=1pt,outer sep=1pt,minimum size=4pt,fill=white] {$\beta_{i-1}$} (y5);
        \draw[thick] (x) to[out=40, in=-100] node[pos=0.6,inner sep=1pt,outer sep=1pt,minimum size=4pt,fill=white] {$\beta_{j-1}$} (y5);
	\draw[thick] (x) to node[midway,inner sep=1pt,outer sep=1pt,minimum size=4pt,fill=white] {$\beta_i$} (y6);
	
	\draw[thick] (x) to node[midway,inner sep=1pt,outer sep=1pt,minimum size=4pt,fill=white] {$\beta_{k-2}$} (y8);
	\end{scope}
	
	\draw[-{Stealth[length=1.6mm]},very thick,decoration = {snake,pre length=3pt,post length=7pt,},decorate] (2.9,1) to node[midway,anchor=south]{$\Shift$} (5,1);
	
	\begin{scope}[xshift=8.3cm]
	\node[circle,fill=black,draw,inner sep=0pt,minimum size=4pt] (x) at (0,0) {};
	\node[anchor=north] at (x) {$x$};
	
	\coordinate (O) at (0,0);
	\def\radius{2.6cm}
	
	\node[circle,fill=black,draw,inner sep=0pt,minimum size=4pt] (y0) at (190:\radius) {};
	\node at (190:2.9) {$y_0$};
	
	\node[circle,fill=black,draw,inner sep=0pt,minimum size=4pt] (y1) at (165:\radius) {};
	\node at (165:2.9) {$y_1$};
	
	\node[circle,fill=black,draw,inner sep=0pt,minimum size=4pt] (y2) at (140:\radius) {};
	\node at (140:2.9) {$y_2$};
	
	\node[circle,fill=black,draw,inner sep=0pt,minimum size=4pt] (y4) at (90:\radius) {};
	\node at (90:2.9) {$y_{i-1}$};
	
	\node[circle,fill=black,draw,inner sep=0pt,minimum size=4pt] (y5) at (62:\radius) {};
	\node at (62:2.9) {$y_i = y_j$};
	
	\node[circle,fill=black,draw,inner sep=0pt,minimum size=4pt] (y6) at (30:\radius) {};
	\node at (30:3) {$y_{i+1}$};
	
	\node[circle,fill=black,draw,inner sep=0pt,minimum size=4pt] (y8) at (-10:\radius) {};
	\node at (-10:3.1) {$y_{k-1}$};
	
	\node[circle,inner sep=0pt,minimum size=4pt] at (115:2) {$\ldots$}; 
	\node[circle,inner sep=0pt,minimum size=4pt] at (12:2) {$\ldots$}; 
	
	\draw[thick] (x) to node[midway,inner sep=1pt,outer sep=1pt,minimum size=4pt,fill=white] {$\beta_0$} (y0);
	\draw[thick] (x) to node[midway,inner sep=1pt,outer sep=1pt,minimum size=4pt,fill=white] {$\beta_1$} (y1);
	\draw[thick] (x) to node[midway,inner sep=1pt,outer sep=1pt,minimum size=4pt,fill=white] {$\beta_2$} (y2);
	
	\draw[thick] (x) to node[pos=0.55,inner sep=1pt,outer sep=1pt,minimum size=4pt,fill=white] {$\beta_{i-1}$} (y4);
	\draw[thick] (x) to[out=80, in=-140] node[pos=0.75,inner sep=1pt,outer sep=1pt,minimum size=4pt,fill=white] {$\beta_{i}$} (y5);
        \draw[thick] (x) to[out=40, in=-100] node[pos=0.6,inner sep=1pt,outer sep=1pt,minimum size=4pt,fill=white] {$\beta_{j}$} (y5);
	\draw[thick] (x) to node[midway,inner sep=1pt,outer sep=1pt,minimum size=4pt,fill=white] {$\beta_{i+1}$} (y6);
	
	\draw[thick, dotted] (x) to (y8);
	\end{scope}
	\end{tikzpicture}
	\caption{The process of shifting a fan.}\label{fig:fan}
\end{figure}
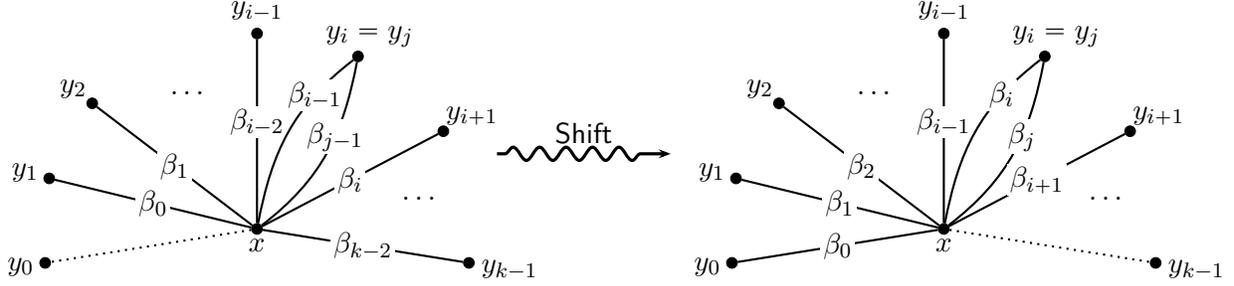

\section{Proof of Theorem~\ref{theo:main_theo}\eqref{theo:shannon}}\label{section:shannon}

In this section, we will assume that $|f(x, L)| \geq \Shav{x}$ for each $x \in V(G)$.
Furthermore, we assume that $|f(x, L)| \leq \ell$ for some $\ell \in \N$.
As mentioned earlier, we will describe a procedure to improve a partial coloring by computing a fan chain that satisfies certain properties.
To this end, we define the \hyperref[alg:shannon_chain]{Local Shannon Fan Algorithm} which takes as input a partial $L$-edge-coloring $\phi$ and outputs the desired fan chain.
Let us provide an informal overview of the algorithm (the full details are provided in Algorithm~\ref{alg:shannon_chain}).
Let $x, y$ be the vertices in $V(e)$ such that $\deg(x) \leq \deg(y)$.
We first check whether $A(\phi, y) \setminus U(\phi, x) \neq \0$, which would imply $e$ is $\phi$-happy.
We may now assume $A(\phi, y) \subseteq U(\phi, x)$.
Let $\eta \in A(\phi, y)$ and let $f$ be the unique edge incident to $x$ colored $\eta$.
We return the fan $F = (e, f)$.

Assuming $A(\cdot)$ and $U(\cdot)$ are stored as hash maps, the \textsf{if} statement takes $O(\Delta\,\ell)$ time and all other operations take $O(\Delta)$ time.
It follows that the running time of Algorithm~\ref{alg:shannon_chain} is $O(\Delta\,\ell)$.

\begin{algorithm}[h]\small
\caption{Local Shannon Fan Algorithm}\label{alg:shannon_chain}
\begin{flushleft}
\textbf{Input}: A proper partial $L$-edge-coloring $\phi$ and an uncolored edge $e \in E(G)$. \\
\textbf{Output}: A fan chain $F$.
\end{flushleft}
\begin{algorithmic}[1]
    \State Let $V(e) = \set{x, y}$ such that $\deg(x) \leq \deg(y)$.
    \If{$A(\phi, y) \setminus U(\phi, x) \neq \0$}
        \State \Return $(e)$ \label{step:happy_edge}
    \EndIf
    \State $\eta \gets \min A(\phi, y)$
    \State Let $f \in E_G(x)$ be such that $\phi(f) = \eta$.
    \State \Return $(e, f)$
\end{algorithmic}
\end{algorithm}

\begin{Lemma}\label{lemma:fan_shannon}
    Let $F$ be the output of Algorithm~\ref{alg:shannon_chain} on input $(\phi, e)$. Then, 
    \begin{enumerate}
        \item either $F$ is $\phi$-happy, or
        \item $F$ is $\phi$-content, or
        \item $A(G, \phi) = A(G, \Shift(\phi, F))$ and $\Shift(\phi, F)$ contains a path chain satisfying the conditions of Lemma~\ref{lemma:path}, or
        \item\label{item:e_succ} $\phi$ contains a path chain satisfying the conditions of Lemma~\ref{lemma:path}.
    \end{enumerate}
\end{Lemma}

\begin{proof}\stepcounter{ForClaims}\renewcommand{\theForClaims}{\ref{lemma:fan_shannon}}
    If $F = (e)$, we must be at step~\ref{step:happy_edge} and so $F$ is $\phi$-happy.
    Suppose $F = (e, f)$ and let $x,\,y,\,\eta$ be as defined in Algorithm~\ref{alg:shannon_chain} and let $z$ be the vertex in $V(f) \setminus \set{x}$.
    Note that we have $A(\phi, y) \subseteq U(\phi, x)$.
    Let $\psi \defeq \Shift(\phi, F)$.
    We will consider each case separately.
    \begin{enumerate}[label=\ep{\normalfont{}\textbf{Case\arabic*}},labelindent=0pt,wide]
        \item $A(\phi, z) \setminus U(\phi, x) \neq \0$. We claim that $F$ is $\phi$-happy.
        Note the following:
        \[U(\psi, x) = U(\phi, x), \quad A(\phi, z) \subseteq A(\psi, z).\]
        In particular, we have $A(\psi, z) \setminus U(\psi, x) \neq \0$, as claimed.

        \item $A(\phi, z) \subseteq U(\phi, x)$ and $\eta \notin f(z, L)$. In this case, we have
        \[A(\psi, x) = A(\phi, x), \quad A(\psi, y) = A(\phi, y) \setminus \set{\eta}, \quad A(\psi, z) = A(\phi, z).\]
        In particular, $A(G, \psi) < A(G, \phi)$ and so $F$ is $\phi$-content.

        \item\label{case:2} $A(\phi, z) \subseteq U(\phi, x)$, $\eta \in f(z, L)$ and $\deg(z) < \deg(y)$.
        In this case, we have
        \[A(\psi, x) = A(\phi, x), \quad A(\psi, y) = A(\phi, y) \setminus \set{\eta}, \quad A(\psi, z) = A(\phi, z) \cup \set{\eta},\]
        i.e., $A(G, \psi) = A(G, \phi)$. However, we have
        \[D(G, \psi) - D(G, \phi) = \deg(z) - \deg(y) < 0,\]
        and so $F$ is $\phi$-content.

        \item $A(\phi, y) \cup A(\phi, z) \subseteq U(\phi, x)$, $\eta \in f(z, L)$ and $\deg(z) \geq \deg(y)$.
        We note that $A(G, \psi) = A(G, \phi)$ by an identical argument to \ref{case:2}.
        It is now enough to find a path $P$ under the coloring $\phi$ or $\psi$ that satisfies the conditions of Lemma~\ref{lemma:path}.
        The following claim will assist with the proof of this case.
        \begin{claim}\label{claim:nonempty_intersection}
            Suppose $A(\phi, y) \cup A(\phi, z) \subseteq U(\phi, x)$ and $\deg(z) \geq \deg(y)$.
            Then, $A(\phi, y) \cap A(\phi, z) \neq \0$.
        \end{claim}
    
        \begin{claimproof}
            Since $x$ and $y$ are adjacent to uncolored edges, we have
            \[|U(\phi, x)| \leq \deg(x) - 1, \quad |A(\phi, y)| \geq \Shav{y} - (\deg(y) - 1) = \lfloor\deg(y)/2\rfloor + 1.\] 
            Similarly, we have
            \[|A(\phi, z)| \geq \Shav{z} - \deg(z) = \lfloor\deg(z)/2\rfloor.\]
            As $A(\phi, y) \cup A(\phi, z) \subseteq U(\phi, x)$, we can conclude that
            \begin{align*}
                |U(\phi, x)| &\geq |A(\phi, y) \cup A(\phi, z)| \\
                ~\implies \deg(x) - 1 &\geq |A(\phi, y)| + |A(\phi, z)| - |A(\phi, y) \cap A(\phi, z)|  \\
                ~\implies |A(\phi, y) \cap A(\phi, z)| &\geq \lfloor\deg(y)/2\rfloor + \lfloor\deg(z)/2\rfloor - \deg(x) + 2.
            \end{align*}
            Note that we define $y$ such that $\deg(x) \leq \deg(y)$. 
            As $\deg(z) \geq \deg(y) \geq \deg(x)$, we may conclude that
            \begin{align*}
                |A(\phi, y) \cap A(\phi, z)| &\geq \lfloor\deg(x)/2\rfloor + \lfloor\deg(x)/2\rfloor - \deg(x) + 2 \\
                &\geq \deg(x) - 1 - \deg(x) + 2 = 1,
            \end{align*}
            as desired.
        \end{claimproof}
        Let $\beta \in A(\phi, y) \cap A(\phi, z)$ and $\alpha \in A(\phi, x)$ be arbitrary.
        An identical argument to the one in \cite[Lemma~2.11]{ShannonChain} shows that for $P \defeq P(e; \phi, \alpha\beta)$ and $P' \defeq P(f; \psi, \alpha\beta)$, either $\vend(P) \neq \vstart(P)$ or $\vend(P') \neq \vstart(P')$.
        In particular, one of the paths satisfies the conditions of Lemma~\ref{lemma:path} under the corresponding coloring.
        This completes the proof. \qedhere
    \end{enumerate}
\end{proof}
 In the terminology of the coloring procedure outlined in \S\ref{section:prelim}, the chain $C$ is either the fan $F$ output by Algorithm~\ref{alg:shannon_chain}, or the single edge $(e)$ in case we are in the situation of \ref{item:e_succ}.

\section{Proof of Theorem~\ref{theo:main_theo}\eqref{theo:vizing}}\label{section:vizing}

In this section, we will assume that $|f(x, L)| \geq \deg(x) + \mu(x)$ for each $x \in V(G)$.
We will also once again assume that $|f(x, L)| \leq \ell$ for some $\ell \in \N$.
As mentioned earlier, we will describe a procedure to improve a partial coloring by computing a fan chain that satisfies certain properties.
To this end, we define the \hyperref[alg:vizing_chain]{Local Vizing Fan Algorithm} which takes as input a partial $L$-edge-coloring $\phi$ and whose output helps compute the desired fan chain.
The algorithm is identical to the one in \cite[Algorithm~A.1]{ShannonChain} with minor edits to account for locality.
Let us provide an informal overview of the algorithm (the full details are provided in Algorithm~\ref{alg:vizing_chain}).
The algorithm takes as input a proper partial coloring $\phi$, an uncolored edge $e \in E(G)$, and a choice of a pivot vertex $x \in V(e)$.
The output is a tuple $(F, \beta, j)$ such that:
\begin{itemize}
    \item $F$ is a fan with $\Start(F) = e$ and $\Pivot(F) = x$,
    \item $\beta \in \N$ is a color and $j$ is an index such that $\beta \in A(\phi, \vend(F))\cap A(\phi, \vend(F|j))$.
\end{itemize}
We first assign $\beta(z) = A(\phi, z)$ to each $z \in N_G(x)$. 
To construct $F$, we follow a series of iterations. 
At the start of each iteration, we have a fan $F = (e_0, \ldots, e_k)$ where $e_0 = e$ and $x \in V(e_i)$ for all $i$. 
Let $y_i \in V(e_i)$ be distinct from $x$ (the $y_i$'s are not necessarily distinct).
If $\min \beta(y_k) \notin U(\phi, x)$, then $F$ is $\phi$-happy and we return $(F, \min \beta(y_k), k+1)$.
If not, let $f \in E_G(x)$ be the unique edge adjacent to $x$ such that $\phi(f) = \min \beta(y_k)$. 
We now have two cases.
\begin{enumerate}[label=\ep{\emph{Case \arabic*}},labelindent=5pt,leftmargin=*]
    \item $f \notin \set{e_0, \ldots, e_k}$. Then we update $F$ to $(e_0, \ldots, e_k, f)$, remove $\min \beta(y_k)$ from $\beta(y_k)$ and continue.
    \item $f = e_j$ for some $0 \leq j \leq k$. Note that $\phi(e_0) = \blank$ and $\min \beta(y_k) \in A(\phi, y_k)$, so we must have $1 \leq j < k$. 
    In this case, we return $(F, \min \beta(y_k), j)$.
\end{enumerate}
We remark that $|A(\phi, z)| \geq \mu(z)$ and so $\beta(z)$ is never empty when defining $\eta$ at step~\ref{step:beta_empty}.
Furthermore, a similar argument as that of Algorithm~\ref{alg:shannon_chain} shows that the running time of Algorithm~\ref{alg:vizing_chain} is $O(\Delta^2\,\ell)$.

\begin{algorithm}[h]\small
\caption{Local Vizing Fan Algorithm}\label{alg:vizing_chain}
\begin{flushleft}
\textbf{Input}: A proper partial $L$-edge-coloring $\phi$, an uncolored edge $e\in E(G)$, and a vertex $x \in V(e)$. \\
\textbf{Output}: A fan $F$ with $\Start(F) = e$ and $\Pivot(F) = x$, a color $\beta \in \N$, and an index $j$ such that $\beta \in A(\phi, \vend(F))\cap A(\phi, \vend(F|j))$.
\end{flushleft}

\begin{algorithmic}[1]
    \State Let $y \in V(e)$ be distinct from $x$.
    \State $\mathsf{nbr}(\eta) \gets \blank$ \textbf{for each} $\eta \in U(\phi, x)$, \quad $\beta(z) \gets A(\phi, z)$ \textbf{for each} $z \in N_G(x)$
    \For{$f\in E_G(x)$}
        \State $\mathsf{nbr}(\phi(f)) \gets f$
    \EndFor
    \medskip
    \State $\mathsf{index}(f) \gets \blank$ \textbf{for each} $f \in E_G(x)$
    \State $F \gets (e)$, \quad $k \gets 0$, \quad $y_k \gets y$, \quad $\mathsf{index}(e) \gets k$
    \While{$k < \deg_G(x)$}
        \State $\eta \gets \min\beta(y_k)$, \quad $\mathsf{remove}(\beta(y_k), \eta)$ \label{step:beta_empty}
        \If{$\eta \notin U(\phi, x)$}
            \State \Return $(F, \eta, k + 1)$ \label{step:happy_fan_vizing}
        \EndIf
        \State $k \gets k + 1$
        \State $e_k \gets \mathsf{nbr}(\eta)$
        \If{$\mathsf{index}(e_k) \neq \blank$}
            \State \Return $(F, \eta, \mathsf{index}(e_k))$
        \EndIf
        \State $\mathsf{index}(e_k) \gets k$, \quad $y_k \in V(e_k)$ distinct from $x$
        \State $\mathsf{append}(F, e_k)$
    \EndWhile
\end{algorithmic}
\end{algorithm}

\begin{Lemma}\label{lemma:fan_vizing}
    Let $(F, \beta, j)$ be the output of Algorithm~\ref{alg:vizing_chain} on input $(\phi, e, x)$. Then for $F' \defeq F|j$, 
    \begin{enumerate}
        \item either $F$ is $\phi$-happy, or
        \item $F$ or $F'$ is $\phi$-content, or
        \item $A(G, \phi) = A(G, \Shift(\phi, F))$ and $\Shift(\phi, F)$ contains a path chain satisfying the conditions of Lemma~\ref{lemma:path}, or
        \item $A(G, \phi) = A(G, \Shift(\phi, F'))$ and $\Shift(\phi, F')$ contains a path chain satisfying the conditions of Lemma~\ref{lemma:path}.
    \end{enumerate}
\end{Lemma}

\begin{proof}
    If we reach step~\ref{step:happy_fan_vizing}, then $F$ is $\phi$-happy.
    Let $\psi \defeq \Shift(\phi, F)$, $\psi' \defeq \Shift(\phi, F')$ and $\alpha \in A(\phi, x)$ be arbitrary.
    shows that for $P \defeq P(\End(F); \psi, \alpha\beta)$ and $P' \defeq P(\End(F'); \psi', \alpha\beta)$, either $\vend(P) \neq \vstart(P)$ or $\vend(P') \neq \vstart(P')$.
    In particular, one of the paths satisfies the conditions of Lemma \ref{lemma:path}.
    It remains to show $A(G, \psi) \leq A(G, \phi)$ and $A(G, \psi') \leq A(G, \phi)$.
    We will just consider the former as the latter follows identically, \textit{mutatis mutandis}.

    It is easy to see that $A(\psi, x) = A(\phi, x)$. 
    Let $F = (e_0, \ldots, e_{k-1})$ and $z$ be a vertex in $F$ other than $x$.
    Consider the set $I(F, z) \defeq \set{0 \leq i < k\,:\, z \in V(e_i)}$.
    We have the following:
    \[A(\psi, z) \subseteq A(\phi, z) \setminus \set{\phi(e_{i+1})\,:\, i \in I(F, z) \setminus \set{k-1}} \cup \set{\phi(e_{i})\,:\, i \in I(F, z)\setminus \set{0}}.\]
    As $\phi(e_{i+1})$ was chosen to be in $A(\phi, z)$, we have
    \[\set{\phi(e_{i+1})\,:\, i \in I(F, z)} \cap \set{\phi(e_{i})\,:\, i \in I(F, z)} = \0,\]
    and
    \[|A(\psi, z)| \leq |A(\phi, z)| - |\set{\phi(e_{i+1})\,:\, i \in I(F, z)}| + |\set{\phi(e_{i})\,:\, i \in I(F, z)}|.\]
    If $\set{0, k-1} \cap I(F, z) = \0$, we have
    \[|\set{\phi(e_{i+1})\,:\, i \in I(F, z)}| = |\set{\phi(e_{i})\,:\, i \in I(F, z)}|.\]
    From here, we may conclude that 
    \[|A(\psi, z)| \leq |A(\phi, z)|, \quad \text{ if } \set{0, k-1} \cap I(F, z) = \0.\]
    Let $0 \in I(F, z)$ and $k - 1 \notin I(F, z)$.
    In this case, we have $|A(\psi, z)| \leq |A(\phi, z)| - 1$ as
    \[|\set{\phi(e_{i+1})\,:\, i \in I(F, z)}| = |\set{\phi(e_{i})\,:\, i \in I(F, z)}| + 1,\]
    since $\phi(e_0) = \blank$.
    Let $k-1\in I(F, z)$ and $0 \notin I(F, z)$.
    In this case, we have $|A(\psi, z)| \leq |A(\phi, z)| + 1$ as
    \[|\set{\phi(e_{i+1})\,:\, i \in I(F, z)}| = |\set{\phi(e_{i})\,:\, i \in I(F, z)}| - 1,\]
    since $e_k$ does not exist.
    Finally, in the case that $\set{0, k-1} \subseteq I(F, z)$ we have $|A(\psi, z)| \leq |A(\phi, z)|$ as
    \[|\set{\phi(e_{i+1})\,:\, i \in I(F, z)}| = |\set{\phi(e_{i})\,:\, i \in I(F, z)}|,\]
    since $\phi(e_0) = \blank$ and $e_k$ does not exist.
    Putting all cases together, we conclude that $A(G, \psi) \leq A(G, \phi)$, as desired.
\end{proof}
In the terminology of the coloring procedure outlined in \S\ref{section:prelim}, the chain $C$ is either the fan $F$ or $F|j$, where $(F, \beta, j)$ is the output of Algorithm~\ref{alg:vizing_chain}.

\section*{Acknowledgements}

We thank Luke Postle and Michelle Delcourt for helpful comments.
We are also grateful to the anonymous referees for carefully reading the manuscript and providing helpful suggestions.

\vspace{0.2in}
\printbibliography

\end{document}